\documentclass[10pt]{article}
\usepackage{eurosym}
\usepackage{hyperref}
\usepackage{amssymb}
\usepackage{amsfonts}
\usepackage{graphicx}
\usepackage{amsmath}
\usepackage{float,color,ulem}
\usepackage{epsf,epsfig,subfigure}
\usepackage{float,color,ulem}
\usepackage{ulem}
\usepackage{bbm}
\usepackage{mathrsfs}
\usepackage[toc,page,title,titletoc,header]{appendix}

\newfloat{figure}{H}{lof}
\newfloat{table}{H}{lot}
\floatname{figure}{\figurename}
\floatname{table}{\tablename}
\newtheorem{theorem}{Theorem}[section]

\newtheorem{assumption}[theorem]{Assumption}

\newtheorem{definition}[theorem]{Definition}

\newtheorem{lemma}[theorem]{Lemma}

\newtheorem{remark}[theorem]{Remark}

\hypersetup{colorlinks=true, linkcolor=black, citecolor=black}
\newenvironment{proof}[1][Proof]{\noindent\textit{#1.} }{\hfill \rule{0.5em}{0.5em}}
\numberwithin{equation}{section}
\newcommand{\dint}{\displaystyle\int}

\begin{document}

\title{\textbf{A system of state-dependent delay differential equation
modelling forest growth II: boundedness of solutions}}
\author{ \textsc{Pierre Magal and Zhengyang Zhang \footnote{The research of this author is supported by China Scholarship Council.} } \\
{\small Univ. Bordeaux, IMB, UMR 5251, F-33076 Bordeaux, France}\\
{\small CNRS, IMB, UMR 5251, F-33400 Talence, France}}
\maketitle
\textbf{Abstract:} {\small In this article we consider a class of state-dependent delay differential equations which is modelling the dynamics of the number of adult trees in forests. We prove the boundedness and the dissipativity of the solutions for an $n$-species model.}\newline
\textbf{AMS Subject Classication} :  	34K05,  	37L99,  37N25.

\section{Introduction}
In this article we are interested in a state-dependent delay differential equation modelling the growth of forest. Following Magal and Zhang \cite{Magal2}, when the forest is composed of a single species of trees, we have the following system
\begin{equation}\label{EQ1.1}
\left\{ 
\begin{array}{l}
A^{\prime }(t)=-\mu _{A}A(t)+\beta e^{-\mu _{J}\tau (t)}\dfrac{f(A(t))}{f(A(t-\tau (t)))}A(t-\tau (t)),\forall t\geqslant 0, \vspace{0.1cm}\\ 
\dint _{t-\tau (t)}^{t}f(A(\sigma ))d\sigma =\dint _{-\tau _{0}}^{0}f(\varphi (\sigma )) d\sigma ,\forall t\geqslant 0,
\end{array}
\right.  
\end{equation}
with the initial conditions 
\begin{equation*}
A(t)=\varphi (t)\geqslant 0,\forall t\leqslant 0\text{ and }\tau (0)=\tau _{0}\geqslant 0,
\end{equation*}
where $\varphi\geqslant 0$ belongs to 
\begin{equation*}
X_{\alpha }:=\left\{ \phi \in C(-\infty ,0]:e^{-\alpha \vert
.\vert }\phi (.)\in BUC(-\infty ,0]\cap \mathrm{Lip}(-\infty ,0]\right\},
\end{equation*}
which is a Banach space endowed with the norm
\begin{equation*}
\Vert\phi\Vert _{X_{\alpha}}:=\Vert e^{-\alpha\vert .\vert}\phi(.)\Vert _{\infty ,(-\infty ,0]}+\Vert e^{-\alpha\vert .\vert}\phi(.)\Vert _{\mathrm{Lip}(-\infty ,0]},
\end{equation*}
where $BUC(-\infty ,0]$ denotes the space of bounded uniformly continuous functions from $(-\infty ,0]$ to $\mathbb{R}$, and $\mathrm{Lip}(-\infty ,0]$ denotes the space of Lipschitz functions from $(-\infty ,0]$ to $\mathbb{R}$. 

Equation (\ref{EQ1.1}) models the dynamics of the adult population of trees. Here $A(t)$ is the number of adult trees at time $t$, $\tau (t)$ is the time needed by newborns to become adult at time $t$, $\mu _{A}>0$ is the mortality rate of the adult trees, $\mu _{J}>0$ is the mortality rate of the juvenile trees, $\beta >0$ is the birth rate. In the context of forest modelling (see \cite{Magal2}), $f(A(t))$ describes the growth rate of juveniles, and the function $f$ is capturing the effect of the competition for light between adults and juveniles. For mathematical convenience, we will make the following assumption.
\begin{assumption}\label{ASS1.1}
We assume that
\begin{itemize}
\item[(i)]  The coefficients $\mu _{A}>0$, $\mu _{J}>0$, $\beta >0$;
\item[(ii)] The function $f:\mathbb{R}\rightarrow (0,+\infty )$ is Lipschitz continuous and continuously differentiable with 
\begin{equation*}
f(x)>0,\lim_{x\rightarrow +\infty }f(x)=0\text{ and }f^{\prime }(x)\leqslant 0,\forall x\in \mathbb{R}
\end{equation*}
\end{itemize}
\end{assumption}

Actually the system (\ref{EQ1.1}) has been first derived by Smith \cite{Smith1} from a size-structured model of the form
\begin{equation*}
\left\{ 
\begin{array}{l}
A^{\prime }(t)=-\mu _{A}A(t)+f(A(t))j(t,s^{\ast }),\forall t\geqslant 0, \vspace{0.1cm}\\ 
\partial _{t}j(t,s)+f(A(t))\partial _{s}j(t,s)=-\mu _{J}j(t,s),\forall s\in [ s_{-},s^{\ast }], \vspace{0.1cm}\\ 
f(A(t))j(t,s_{-})=\beta A(t), \vspace{0.1cm}\\
A(0)=A_{0}\geqslant 0, \vspace{0.1cm}\\
j(0,s)=j_{0}(s)\geqslant 0, \forall s\in [s_{-},s^{\ast }),
\end{array}
\right.
\end{equation*}
where $0\leqslant s_{-}<s^{\ast }$ are the minimal and maximal size of juveniles, and $j(t,s)$ is the density of juveniles with size $s$ at time $t$. The system (\ref{EQ1.1}) has also been extensively studied by Smith in \cite{Smith1, Smith2,  Smith3, Smith4}, where the author introduced a change of variable to transform this kind of state-dependent delay differential equation into a constant delay differential equation. The change of variable is given by 
\begin{equation*}
x=\int_{0}^{t}f(A(\sigma ))d\sigma =:\Phi (t).
\end{equation*}
Set
\begin{equation*}
\delta :=\int_{-\tau _{0}}^{0}f(\varphi (\sigma )) d\sigma \geqslant 0,
\end{equation*}
then for $x\geqslant \delta$, 
\begin{equation*}
x-\delta =\int_{0}^{t}f(A(\sigma ))d\sigma -\int_{t-\tau (t)}^{t}f(A(\sigma ))d\sigma =\int_{0}^{t-\tau (t)}f(A(\sigma ))d\sigma =\Phi (t-\tau (t)),
\end{equation*}
This means that $x-\delta $ corresponds to $t-\tau (t)$ under this change of variable. Moreover by setting $W(x)=A(t)$ and using the same arguments as in Smith \cite{Smith1}, one also has 
\begin{equation*}
\tau ( t) =\int_{-\delta }^{0}f(W(x+r))^{-1}dr.
\end{equation*}
Therefore Smith \cite{Smith1} obtained the following constant delay differential equation 
\begin{equation}\label{EQ1.2}
W^{\prime }(x)=-\mu _{A}\dfrac{W(x)}{f(W(x))}+\beta e^{-\mu _{J}\int_{-\delta }^{0}f(W(x+r)) ^{-1}dr}\dfrac{W(x-\delta )}{f(W(x-\delta ))},\forall x\geqslant 0.  
\end{equation}
Based on the analysis of this equation (\ref{EQ1.2}), Smith \cite{Smith1, Smith2,  Smith3, Smith4} was able to prove the boundedness of solutions whenever $\delta >0$. Along the same line, he was also able to analyze the uniform persistence and Hopf bifurcation around the positive equilibrium. The result on boundedness of solutions for this case is as follows.
\begin{theorem}\label{TH1.2}
Let Assumption \ref{ASS1.1} be satisfied. Assume that $\tau _{0}>0$. Then for each $\varphi \geqslant 0$ and $\varphi \in X_{\alpha }$, the corresponding solution of equation (\ref{EQ1.1}) is bounded.
\end{theorem}
\begin{remark}\label{RE1.3}
One may observe that the boundedness of solutions might not be true when $\tau _{0}=0$. Indeed, by the second equation of (\ref{EQ1.1}),
\begin{equation*}
\tau _{0}=0\Rightarrow \tau (t)=0,\forall t\geqslant 0,
\end{equation*}
and in this special case the first equation of (\ref{EQ1.1}) becomes linear:
\begin{equation}\label{EQ1.3}
A^{\prime }(t)=(\beta -\mu _{A}) A(t),\forall t\geqslant 0.
\end{equation}
The solution of (\ref{EQ1.3}) exists but when $\beta -\mu _{A}>0$, every strictly positive solution is unbounded.
\end{remark}

Consider now the following $n$-species model
\begin{equation}\label{EQ1.4}
\left\{ 
\begin{array}{l}
A_{i}^{\prime }(t)=-\mu _{A_{i}}A_{i}(t)+\beta _{i}e^{-\mu _{J_{i}}\tau _{i}(t)}\dfrac{f_{i}(Z_{i}(t))}{f_{i}(Z_{i}(t-\tau _{i}(t)))}A_{i}(t-\tau _{i}(t)),\forall t\geqslant 0, \\ 
\dint_{t-\tau _{i}(t)}^{t}f_{i}(Z_{i}(\sigma ))d\sigma =\dint _{-\tau _{i0}}^{0}f_{i}(Z_{i\varphi}(\sigma )) d\sigma ,\forall t\geqslant 0,
\end{array}
\right.   
\end{equation}
with the initial conditions
\begin{equation*}
A_{i}(t)=\varphi _{i}(t)\in X_{\alpha},\varphi _{i}(t)\geqslant 0,\forall t\leqslant 0\text{ and }\tau _{i}(0)=\tau _{i0}\geqslant 0,
\end{equation*}
where
\begin{equation*}
Z_{i}(t)=\sum _{j=1}^{n}\zeta _{ij}A_{j}(t),Z_{i\varphi}(t):=\sum _{j=1}^{n}\zeta _{ij}\varphi _{j}(t)
\end{equation*}
with $\zeta _{ij}\geqslant 0$, $i=1,\dots ,n$. We will use the following assumptions.

\begin{assumption}\label{ASS1.4}
We assume that $\forall i=1,\dots ,n$,

\begin{itemize}
\item[(i)]  The coefficients $\mu _{A_{i}}>0$, $\mu _{J_{i}}>0$, $\beta _{i}>0$ and $\zeta _{ii}>0$;
\item[(ii)] The function $f_{i}$ satisfies Assumption \ref{ASS1.1}-(ii) and 
\begin{equation} \label{EQ1.5}
\sup_{x\geqslant 0}\frac{f_{i}(x)}{f_{i}(cx)}<+\infty ,\forall c\geqslant 1.
\end{equation}
\end{itemize}
\end{assumption}

In this article, we will prove the following result for $n$-species model (\ref{EQ1.4}).
\begin{theorem}\label{TH1.5}
Let Assumption \ref{ASS1.4} be satisfied. Then for each nonnegative initial values $\varphi _{i}\geqslant 0$ and $\varphi _{i}\in X_{\alpha }$ and each $\tau _{i0}>0$, the corresponding solution of equation (\ref{EQ1.4}) is bounded.
\end{theorem}
\begin{remark} 
The proof of Theorem \ref{TH1.2} (single species case) uses a similar argument as the proof of Theorem \ref{TH1.5} ($n$-species case), which will be presented in Section 3. But for the single species case, the condition (\ref{EQ1.5}) in Assumption \ref{ASS1.4} is no longer needed.  
\end{remark}
\begin{remark} 
For the $n$-species case we can no longer use the change of variable employed by Smith in \cite{Smith1, Smith2} since the delays $\tau_i(t)$ are different in general. Nevertheless, in this article we show that the arguments employed to prove the boundedness of solutions and the dissipativity in \cite{Smith1, Smith2} can be adapted to the $n$-species case.  
\end{remark}
\begin{remark} 
It is necessary to assume that $\tau _{i0}>0$ because we possibly have 
\[
\zeta _{ij}=0, \forall i\neq j.
\]
Hence it is necessary to assume that in the case of species without coupling, the solution is bounded.  
\end{remark}

State-dependent delay differential equations have been used by several authors to describe the stage-structured population dynamics. We refer to \cite{Aiello, Alomari, Arino, Hartung, Hbid, Kloosterman} for more results on this topic. We also refer to Walther \cite{Walther} for a very general analysis of the semiflow generated by state-dependent delay differential equations. 

The paper is organized as follows. In Section 2 we will present some results about the delay $\tau(t)$. In Section 3 we will prove the boundedness of solutions of the $n$-species model (\ref{EQ1.4}) without using the change of variable. In Section 4 we prove a dissipativity result for such a system. 

\section{Properties of the integral equation for $\tau(t)$}

For simplicity, we focus on the single species model (\ref{EQ1.1}) in this section. The same result can be similarly deduced for the $n$-species model (\ref{EQ1.4}). We have the following lemma of the equivalence of the integral equation for $\tau(t)$ and an ordinary differential equation.
\begin{lemma}\label{LE2.1} Let $A:(-\infty,r) \to \mathbb{R}$ be a given continuous function with $r > 0$. Then there exists a uniquely determined function $\tau:[0,r)\rightarrow [0,+\infty)$ satisfying 
\begin{equation}  \label{EQ2.1}
\int_{t-\tau (t)}^{t}f(A(\sigma))d\sigma =\int_{-\tau _{0}}^{0}f(\varphi(\sigma ))d\sigma ,\forall t\in [0,r).
\end{equation}
Moreover this uniquely determined function $t\mapsto \tau (t)$ is
continuously differentiable and satisfies the ordinary differential equation 
\begin{equation}  \label{EQ2.2}
\tau ^{\prime }(t)=1-\frac{f(A(t))}{f(A(t-\tau (t)))},\forall t\in [0,r),%
\text{ and }\tau (0)=\tau _{0}.
\end{equation}
Conversely if $t\mapsto \tau(t)$ is a $C^{1}$ function satisfying the above
ordinary differential equation (\ref{EQ2.2}), then it also satisfies the
above integral equation (\ref{EQ2.1}).
\end{lemma}

\begin{remark}
By using equation (\ref{EQ2.2}), it is easy to check that 
\begin{equation*}
\tau _{0}>0\Rightarrow \tau (t)>0,\forall t\in [0,r)
\end{equation*}
and 
\begin{equation*}
\tau _{0}=0 \Rightarrow \tau (t)=0,\forall t\in [0,r).
\end{equation*}
\end{remark}

\begin{proof}
Let $t\in [0,r]$. Since by Assumption \ref{ASS1.1}, $f$ is strictly positive, then by considering the function $\displaystyle\tau\mapsto\int_{t-\tau }^{t}f(A(s))ds$, and observing that 
\begin{equation*}
\int_{t-0}^{t}f(A(s))ds=0<\int_{-\tau _{0}}^{0}f(\varphi (s))ds\text{ and }%
\int_{t-(t+\tau _{0})}^{t}f(A(s))ds\geqslant \int_{-\tau _{0}}^{0}f(\varphi (s))ds,
\end{equation*}
it follows by the intermediate value theorem that there exists a unique $\tau (t)\in [0,t+\tau _{0}]$.

By applying the implicit function theorem to the map $\psi:(t,\gamma) \mapsto \dint_{\gamma }^{t}f(A(s))ds$ (which is possible since $\displaystyle\frac{\partial\psi}{\partial\gamma}=-f(A(\gamma))$ and by Assumption \ref{ASS1.1}, $f$ is strictly positive), we deduce that $t\mapsto t-\tau (t)$ is continuously differentiable, and by computing the derivative with respect to $t$ on both sides of (\ref{EQ2.1}), we deduce that $\tau (t)$ is a solution of (\ref{EQ2.2}).

Conversely, assume that $\tau (t)$ is a solution of (\ref{EQ2.2}). Then 
\begin{equation*}
f(A(t))=(1-\tau ^{\prime }(t)) f(A(t-\tau (t))),\forall t\in [0,r).
\end{equation*}
Integrating both sides with respect to $t$, we have 
\begin{equation*}
\int_{0}^{t}f(A(s))ds=\int_{0}^{t}f(A(s-\tau (s)))\left( 1-\tau ^{\prime
}(s)\right) ds.
\end{equation*}
Make the change of variable $l=s-\tau (s)$, we have $\forall t\in [0,r)$, 
\begin{equation*}
\begin{split}
& \int_{0}^{t}f(A(s))ds=\int_{-\tau _{0}}^{t-\tau (t)}f(A(l))dl \\
\Leftrightarrow & \int_{t-\tau (t)}^{t}f(A(s))ds+\int_{0}^{t-\tau
(t)}f(A(s))ds=\int_{-\tau _{0}}^{t-\tau (t)}f(A(s))ds \\
\Leftrightarrow & \int_{t-\tau (t)}^{t}f(A(s))ds=\int_{-\tau _{0}}^{t-\tau
(t)}f(A(s))ds-\int_{0}^{t-\tau (t)}f(A(s))ds,
\end{split}%
\end{equation*}
this implies that $\tau(t)$ also satisfies the equation (\ref{EQ2.1}).
\end{proof}

In order to see that the delay $\tau (t)$ is a functional of $A_{t} \in X_\alpha$ which is defined as
$$
A_{t}(\theta):=A(t+\theta),\forall \theta \leq 0,
$$
we define the following functional. For any constant $C>0$, we define the map $\widehat{\tau}:D(\widehat{\tau})\subset C(-\infty ,0]\times [0,+\infty )\rightarrow [0,+\infty )$ as the solution of 
\begin{equation}\label{EQ2.3}
\int_{-\widehat{\tau}(\phi ,C)}^{0}f(\phi (s))ds=C  
\end{equation}
and its domain 
\begin{equation*}
D(\widehat{\tau})=\left\{( \phi ,C) \in C((-\infty ,0])\times [0,+\infty ):C<\int_{-\infty }^{0}f(\phi (s))ds\right\} .
\end{equation*}
\begin{lemma}\label{LE2.3} 
Set $\displaystyle C_{0}:=\int_{-\tau _{0}}^{0}f(\varphi (s))ds$, then we have the following relation 
\begin{equation*}
\widehat{\tau}(A_{t},C_{0})=\tau (t),\forall t\in (0,r),
\end{equation*}
where $\tau (t)$ is the solution of (\ref{EQ2.1}).
\end{lemma}
\begin{proof}
It is sufficient to observe that 
\begin{equation*}
\int_{-\widehat{\tau}(A_{t},C_{0})}^{0}f(A_{t}(s))ds=\int_{t-\widehat{\tau}(A_{t},C_{0})}^{t}f(A(s))ds=C_{0}.
\end{equation*}
\end{proof}
\section{Boundedness of solutions for $n$-species case}

In this section we will investigate the boundedness of a trajectory of system (\ref{EQ1.4}) with the initial conditions satisfying 
\[
\dint _{-\tau _{i0}}^{0}f_{i}(Z_{i\varphi}(\sigma )) d\sigma>0, \forall i=1,...,n. 
\]
Without loss of generality, we can assume that
\begin{assumption} \label{ASS3.1}
\begin{equation*}
\dint _{-\tau _{i0}}^{0}f_{i}(Z_{i\varphi}(\sigma )) d\sigma =1, \forall i=1,...,n. 
\end{equation*}
\end{assumption}
We have the following lemma from \cite{Magal2}.
\begin{lemma}\label{LE3.2}
Let Assumptions \ref{ASS1.4} and \ref{ASS3.1} be satisfied. Then the functions $t-\tau _{i}(t)$ are strictly increasing with respect to $t$, $\forall i=1,\dots ,n$.
\end{lemma}
Next we will prove the following result.
\begin{lemma}\label{LE3.3}
Let Assumptions \ref{ASS1.4} and \ref{ASS3.1} be satisfied. Then there exists a positive real number  $t_{i}^{\ast }>0$ such that $t_{i}^{\ast }-\tau _{i}(t_{i}^{\ast })=0$, $\forall i=1,\dots ,n$.
\end{lemma}
\begin{proof}
We define
\begin{equation*}
t_{i}^{\ast }:=\sup\{t\geqslant 0:s-\tau _{i}(s)\leqslant 0,\forall s\in[0,t]\},i=1,\dots ,n.
\end{equation*}
\textbf{Case 1:} We assume that all the elements of $\{t_{i}^{\ast }\}_{i=1}^{n}$ are infinite, and we will prove that this is not possible. By the above definition of $t_{i}^{\ast}$, we have $\forall t\geqslant 0$, $t-\tau _{i}(t)\leqslant 0$, or precisely, 
\begin{equation*}
t-\tau _{i}(t)\in [-\tau _{i0},0].
\end{equation*}
Then the equation for $A_{i}(t)$ becomes 
\begin{equation*}
A_{i}^{\prime }(t)=-\mu _{A_{i}}A_{i}(t)+\beta _{i}e^{-\mu _{J_{i}}\tau
_{i}(t)}\dfrac{f_{i}(Z_{i}(t))}{f_{i}(Z_{i\varphi}(t-\tau _{i}(t)))}\varphi
_{i}(t-\tau _{i}(t)),\forall t\geqslant 0.
\end{equation*}
We set 
\begin{equation*}
\Gamma _{i}:=\beta _{i}\sup_{t\in \lbrack -\tau _{i0},0]}\frac{\varphi _{i}(t)}{f_{i}(Z_{i\varphi}(t))}>0.
\end{equation*}%
Since $f_{i}(Z_{i}(t))\leqslant f_{i}(\zeta _{ii}A_{i}(t)),\forall t\geqslant 0$, then by the comparison principle, we have $A_{i}(t)\leqslant \hat{A}_{i}(t)$, $\forall t\geqslant 0$, where $\hat{A}_{i}(t)$ is the solution of 
\begin{equation*}
\left\{ 
\begin{array}{l}
\hat{A}_{i}^{\prime }(t)=-\mu _{A_{i}}\hat{A}_{i}(t)+\Gamma _{i}f_{i}(\zeta _{ii}\hat{A}_{i}(t))=:g_{\Gamma _{i}}(\hat{A}_{i}(t)),\forall t\geqslant 0, \\ 
\hat{A}_{i}(0)=\varphi _{i}(0)\geqslant 0.
\end{array}
\right.
\end{equation*}
As $g_{\Gamma _{i}}(\hat{A}_{i})$ is decreasing with $\hat{A}_{i}$ and we have 
\begin{equation*}
g_{\Gamma _{i}}(0)=\Gamma _{i}f_{i}(0)>0,\lim_{\hat{A}_{i}\rightarrow +\infty }g_{\Gamma _{i}}(\hat{A}_{i})=-\infty ,
\end{equation*}
so fixing $\hat{A}_{i}^{\ast }\in [\varphi _{i}(0),+\infty )$ such that $g_{\Gamma _{i}}(\hat{A}_{i}^{\ast })\leqslant 0$, we have 
\begin{equation*}
A_{i}(t)\leqslant \hat{A}_{i}(t)\leqslant \hat{A}_{i}^{\ast },\forall t\geqslant 0.
\end{equation*}%
Now since by assumption $t-\tau_i(t)\leq 0,\forall t\geq 0$, we obtain for each $t\geq 0$
\begin{equation}\label{EQ3.1}
1=\int_{t-\tau_i(t)}^{t}f_{i}(Z_{i}(\sigma ))d\sigma\geqslant \int_{0}^{t}f_{i}(Z_{i}(\sigma ))d\sigma\geqslant tf_{i}\left(\sum _{j=1}^{n}\zeta _{ij}\hat{A}_{j}^{\ast }\right)
\end{equation}
which impossible.\\ 
\textbf{Case 2:} We assume that exactly $j$ elements of $\{t_{i}^{\ast}\}_{i=1}^{n}$ are finite, where $1\leqslant j<n$, and we will prove that this is not possible, either. Without loss of generality we might assume that $t_{1}^{\ast},\dots ,t_{j}^{\ast}$ are finite. First we prove that $A_{1}(t),\ldots ,A_{j}(t)$ are bounded.


Following a similar argument as in case 1, for each $i=j+1,\dots ,n$, as $t_{i}^{\ast}$ is infinite, we can find $\hat{A}_{i}^{\ast }\in [\varphi _{i}(0),+\infty )$ such that 
\begin{equation*}
A_{i}(t)\leqslant \hat{A}_{i}^{\ast }.
\end{equation*}
For each $k=1,\dots ,j$, consider the solution
\begin{equation*}
z_{k}(t)=z_{k}(t;m_{k})=m_{k}e^{-\mu _{A_{k}}t},t\geqslant 0
\end{equation*}
of the following ordinary differential equation
\begin{equation}\label{EQ3.2}
z_{k}^{\prime }(t)=-\mu _{A_{k}}z_{k}(t),z_{k}(0)=m_{k}.
\end{equation}
We define $\tau _{k,m_{k}}>0$ satisfying
\begin{equation*}
\int _{0}^{\tau _{k,m_{k}}}f_{k}(\zeta _{kk}z_{k}(\sigma ))d\sigma =1.
\end{equation*}
When $\zeta _{kk}>0$ (as is assumed in Assumption \ref{ASS1.4}-(i)), since
\begin{equation*}
\int_{0}^{\tau }f_{k}(\zeta _{kk}z_{k}(\sigma ))d\sigma \geqslant \int_{0}^{\tau }f_{k}(\zeta _{kk}m_{k})d\sigma =\tau f_{k}(\zeta _{kk}m_{k})>0 \text{ when }\tau >0,
\end{equation*}
then $\tau _{k,m_{k}}>0$ exists and is finite. Next we observe that we have  
\begin{equation}\label{EQ3.3}
\tau _{k,m_{k}}\rightarrow +\infty \text{ as } m_{k}\rightarrow +\infty .
\end{equation}
Indeed, assume by contradiction that there exists a subsequence $\lbrace m_{k} \rbrace _{k \geq 0}\rightarrow +\infty$ and a sequence $\lbrace \tau _{k,m_{k}} \rbrace _{k \geq 0}$ bounded by $\tau^\star >0$. Then we have 
$$
1=\int _{0}^{\tau _{k,m_{k}}}f_{k}(\zeta _{kk}z_{k}(\sigma ))d\sigma \leq \int _{0}^{\tau^\star}f_{k}(\zeta _{kk}z_{k}(\sigma ))d\sigma \to 0 \text{ as } k \to +\infty
$$
impossible. \\
By Assumption \ref{ASS1.4}-(ii), for each $\forall c\geqslant 1$
\begin{equation*}
M_{f_{k}}(c):=\sup _{x\geqslant 0}\frac{f_{k}(x)}{f_{k}(cx)}<+\infty.
\end{equation*}
By using \eqref{EQ3.3} we can fix $m_{k}$ (large enough) such that  
\begin{equation*}
-\mu _{A_{k}}+\beta _{k}e^{-\mu _{J_{k}}\tau _{k,m_{k}}}M_{f_{k}}\left(\frac{\zeta _{k1}+\dots +\zeta _{kn}}{\zeta _{kk}}\right)<0.
\end{equation*}
For a constant $K>0$, define
\begin{equation*}
\tilde{t}:=\sup\{t\geqslant 0:\max\{A_{1}(s),\dots ,A_{j}(s)\}\leqslant K,\forall s\in [0,t]\}.
\end{equation*}
In order to prove the boundedness of $A_{1}(t),\ldots ,A_{j}(t)$, we assume by contradiction that $\tilde{t}$ is finite, then at least one of $A_{k}(t),k=1,\dots ,j$ reaches $K$ at $\tilde{t}$. We assume that $A_{1}(\tilde{t})=K$.  Let us prove that 
\begin{equation}\label{3.4}
\tilde{t}-\tau_1(\tilde{t})>0
\end{equation}
for each $K>0$ large enough. Otherwise using the same comparison principle arguments as in the case 1, we can prove that 
$$
K=A_{1}(\tilde{t})\leq \hat{A}_{1}^{\ast }
$$
which is impossible when $K$ is large enough. \\

Now we will prove $A_{1}(\tilde{t}-\tau _{1}(\tilde{t}))\rightarrow +\infty$ when $K\rightarrow +\infty$. As $\tilde{t}$ is finite, and by construction $A_{1}(t)\leqslant K$, $\forall t\in[0,\tilde{t}]$ we must have $A_{1}^{\prime }(\tilde{t})\geqslant 0$. Then
\begin{eqnarray*}
0 & \leqslant & A_{1}^{\prime}(\tilde{t})=-\mu _{A_{1}}A_{1}(\tilde{t})+\beta _{1}e^{-\mu _{J_{1}}\tau _{1}(\tilde{t})}\frac{f_{1}(Z_{1}(\tilde{t}))}{f_{1}(Z_{1}(\tilde{t}-\tau _{1}(\tilde{t})))}A_{1}(\tilde{t}-\tau _{1}(\tilde{t})) \\
& \leqslant & -\mu _{A_{1}}K+\beta _{1}\frac{f_{1}(\zeta _{11}K)}{f_{1}((\zeta _{11}+\dots +\zeta _{1n})\hat{K})}A_{1}(\tilde{t}-\tau _{1}(\tilde{t})),
\end{eqnarray*}
where
\begin{equation*}
\hat{K}:=\max\left\{K,\hat{A}_{j+1}^{\ast},\dots ,\hat{A}_{n}^{\ast},\max _{t\in [-\tau _{10},0]}\varphi _{1}(t),\dots ,\max _{t\in [-\tau _{n0},0]}\varphi _{n}(t)\right\}.
\end{equation*}
Notice that $\displaystyle\frac{(\zeta _{11}+\dots +\zeta _{1n})\hat{K}}{\zeta _{11}K}>1$, then
\begin{equation*}
A_{1}(\tilde{t}-\tau _{1}(\tilde{t}))\geqslant \frac{\mu _{A_{1}}}{\beta _{1}}\cdot\frac{f_{1}((\zeta _{11}+\cdots +\zeta _{1n})\hat{K})}{f_{1}(\zeta _{11}K)}\geqslant \frac{\mu _{A_{1}}K}{\beta _{1}}\cdot\frac{1}{M_{f_{1}}\left(\frac{(\zeta _{11}+\cdots +\zeta _{1n})\hat{K}}{\zeta _{11}K}\right)}.
\end{equation*}
Now since for all $K>0$ large enough $\hat{K}=K$, we deduce that 
$$
A_{1}(\tilde{t}-\tau _{1}(\tilde{t}))\rightarrow +\infty \text{ as } K\rightarrow +\infty. 
$$ 
By using \eqref{3.4}, we can fix $K$ large enough such that 
$$
A_{1}(\tilde{t}-\tau _{1}(\tilde{t}))\geqslant m_{1} \text{ and } \tilde{t}-\tau _{1}(\tilde{t}) \geq 0.
$$
By using the comparison principle on the equation (\ref{EQ3.2}) and
\begin{equation*}
A_{1}^{\prime }(t)\geqslant -\mu _{A_{1}}A_{1}(t),\forall t\geqslant \tilde{t}-\tau _{1}(\tilde{t})
\end{equation*}
with
\begin{equation*}
A_{1}(\tilde{t}-\tau _{1}(\tilde{t}))\geqslant m_{1},
\end{equation*}
we have 
\begin{equation*}
A_{1}(t)\geqslant z_{1}(t-\tilde{t}+\tau _{1}(\tilde{t})),\forall t\geqslant \tilde{t}-\tau _{1}(\tilde{t}).
\end{equation*}
An integration shows that 
\begin{eqnarray*}
1 &=&\int_{\tilde{t}-\tau _{1}(\tilde{t})}^{\tilde{t}}f_{1}(Z_{1}(\sigma ))d\sigma \leqslant \int_{\tilde{t}-\tau _{1}(\tilde{t})}^{\tilde{t}}f_{1}(\zeta _{11}A_{1}(\sigma ))d\sigma \\
&\leqslant &\int_{\tilde{t}-\tau _{1}(\tilde{t})}^{\tilde{t}}f_{1}(\zeta _{11}z_{1}(\sigma -\tilde{t}+\tau _{1}(\tilde{t})))d\sigma =\int_{0}^{\tau _{1}(\tilde{t})}f_{1}(\zeta _{11}z_{1}(\sigma ))d\sigma .
\end{eqnarray*}
By the definition of $\tau _{1,m_{1}}$, we have 
\begin{equation*}
\tau _{1}(\tilde{t})\geqslant \tau _{1,m_{1}}.
\end{equation*}
Now we have
\begin{eqnarray*}
0& \leqslant & A_{1}^{\prime }(\tilde{t})=-\mu _{A_{1}}A_{1}(\tilde{t})+\beta _{1}e^{-\mu _{J_{1}}\tau _{1}(\tilde{t})}\frac{f_{1}(Z_{1}(\tilde{t}))}{f_{1}(Z_{1}(\tilde{t}-\tau _{1}(\tilde{t})))}A_{1}(\tilde{t}-\tau _{1}(\tilde{t})) \\
& = & f_{1}(Z_{1}(\tilde{t}))\left[-\mu _{A_{1}}\frac{A_{1}(\tilde{t})}{f_{1}(Z_{1}(\tilde{t}))}+\beta _{1}e^{-\mu _{J_{1}}\tau _{1}(\tilde{t})}\frac{A_{1}(\tilde{t}-\tau _{1}(\tilde{t}))}{f_{1}(Z_{1}(\tilde{t}-\tau _{1}(\tilde{t})))} \right] \\
& \leqslant & f_{1}(Z_{1}(\tilde{t}))\left[-\mu _{A_{1}}\frac{K}{f_{1}(\zeta _{11}K)}+\beta _{1}e^{-\mu _{J_{1}}\tau _{1,m_{1}}}\frac{K}{f_{1}((\zeta _{11}+\dots +\zeta _{1n})K)}\right] \\
& = & \frac{f_{1}(Z_{1}(\tilde{t}))K}{f_{1}(\zeta _{11}K)}\left[-\mu _{A_{1}}+\beta _{1}e^{-\mu _{J_{1}}\tau _{1,m_{1}}}\frac{f_{1}(\zeta _{11}K)}{f_{1}((\zeta _{11}+\dots +\zeta _{1n})K)}\right] \\
& \leqslant & \frac{f_{1}(Z_{1}(\tilde{t}))K}{f_{1}(\zeta _{11}K)}\left[-\mu _{A_{1}}+\beta _{1}e^{-\mu _{J_{1}}\tau _{1,m_{1}}}M_{f_{1}}\left(\frac{\zeta _{11}+\dots +\zeta _{1n}}{\zeta _{11}}\right)\right]<0,
\end{eqnarray*}
which leads to a contradiction. Thus $\tilde{t}=+\infty$, namely 
$$
A_{k}(t)\leqslant K, \forall t\geqslant 0,\forall  k=1,\dots ,j.
$$
To conclude the proof it remains to observe that for each $i=j+1,\dots ,n$, we have the following formula similar to (\ref{EQ3.1}):
\begin{eqnarray*}
1&=& \int_{t-\tau_i(t)}^{t}f_{i}(Z_{i}(\sigma ))d\sigma \geqslant \int _{0}^{t}f_{i}(Z_{i}(\sigma ))d\sigma \\ &\geqslant&  \int _{0}^{t}f_{i}((\zeta _{i1}+\dots +\zeta _{ij})K+\zeta _{i,j+1}\hat{A}_{j+1}^{\ast}+\dots +\zeta _{in}\hat{A}_{n}^{\ast})d\sigma \\
& = & tf_{i}((\zeta _{i1}+\dots +\zeta _{ij})K+\zeta _{i,j+1}\hat{A}_{j+1}^{\ast}+\dots +\zeta _{in}\hat{A}_{n}^{\ast})
\end{eqnarray*}
which is impossible when $t$ is large enough. 
\end{proof}

Now we turn to the proof of Theorem \ref{TH1.5}.

\begin{proof}[Proof of Theorem \ref{TH1.5}]
For each $i=1,\dots ,n$, we define $\tau _{i,m_{i}}$ satisfying
\begin{equation*}
\int_{0}^{\tau _{i,m_{i}}}f_{i}(\zeta _{ii}z_{i}(\sigma ))d\sigma =1,
\end{equation*}
where $z_{i}(t)=m_{i}e^{-\mu _{A_{i}}t}$, $t\geqslant 0$. As before, we can find $m_{i}$ large enough such that
\begin{equation*}
\beta _{i}e^{-\mu _{J_{i}}\tau _{i,m_{i}}}M_{f_{i}}\left(\frac{\zeta _{i1}+\dots +\zeta _{in}}{\zeta _{ii}}\right)<\mu _{A_{i}}.
\end{equation*}
For a constant $K>0$, we define
\begin{equation*}
\tilde{t}:=\sup\{t>0:\max\{A_{1}(s),\dots ,A_{n}(s)\}\leqslant K,\forall s\in [0,t]\}.
\end{equation*}
Then similar to the proof of case 2 in Lemma \ref{LE3.3}, we can get a fixed $K$ large enough and we can deduce that $\tilde{t}=+\infty$. Thus $A_{i}(t)$ is bounded, $\forall t\geqslant 0$.
\end{proof}

\section{Dissipativity of the system}

In this section we will investigate the dissipativity of the system (\ref{EQ1.4}). First we will prove the following lemma.
\begin{lemma} \label{LE4.1} 
Let Assumptions \ref{ASS1.4} and \ref{ASS3.1} be satisfied. Suppose that $\tau _{i}(t)$ is the solution of (\ref{EQ1.4}), then
\begin{equation*}
\lim _{t\rightarrow +\infty}[t-\tau _{i}(t)]=+\infty .
\end{equation*}
\end{lemma}
\begin{proof} 
If $\tau _{i0}=0$, then $\tau _{i}(t)=0, \forall t \geqslant 0$, and the above result holds naturally. For $\tau _{i0}>0$ we can apply Lemma \ref{LE3.3}, which shows that there exists $t_{i}^{\ast}>0$ such that $t_{i}^{\ast}-\tau _{i}(t_{i}^{\ast})=0$. And by Lemma \ref{LE3.2}, $\forall t\geqslant t_{i}^{\ast}$, $t-\tau _{i}(t)\geqslant 0$, which means that $t-\tau _{i}(t)$ can cross 0 and go above.

Next, for any fixed $\hat{t}>0$, let $\hat{\varphi}_{i}=A_{i,\hat{t}}$ and $\hat{\tau}_{i0}=\tau _{i}(\hat{t})$, then by the semiflow property, the semiflow defined by system (\ref{EQ1.4}) with the new initial conditions $\hat{\varphi}_{i}$ and $\hat{\tau} _{i0}$ will be a translation of the semiflow of system (\ref{EQ1.4}) with the original initial conditions $\varphi _{i}$ and $\tau_{i0}$. Then we can repeat the previous proof of Lemma \ref{LE3.2}, thus we can find a time $\hat{t}_{i}^{\ast}>0$ such that 
\begin{equation*}
\hat{t}_{i}^{\ast}-\hat{\tau} _{i}(\hat{t}_{i}^{\ast})=0,
\end{equation*}
where $\hat{\tau}_{i}(t)$ is the solution of system (\ref{EQ1.4}) under the new initial conditions $\hat{\varphi}_{i}$ and $\hat{\tau} _{i0}$, which satisfies $\hat{\tau} _{i}(t)=\tau _{i}(t+\hat{t})$, $\forall t\geqslant 0$. Hence, we have
\begin{equation*}
(\hat{t}_{i}^{\ast}+\hat{t})-\tau _{i}(\hat{t}_{i}^{\ast}+\hat{t})=\hat{t}.
\end{equation*}
So for any $\hat{t}>0$, we can find $\hat{t}_{i}^{\ast}+\hat{t}>0$ such that $\forall t\geqslant\hat{t}_{i}^{\ast}+\hat{t}$, $t-\tau _{i}(t)\geqslant\hat{t}$. This completes the proof.
\end{proof}
\begin{definition} \label{DE4.2} 
For each real number $\delta \geqslant 0$ we define 
\begin{equation*}
D_{\delta}:=\left\{ (\varphi,\tau_0) \in X_\alpha ^{n} \times \left[ 0,+\infty \right) ^{n}: \dint _{-\tau _{i0}}^{0}f_{i}(Z_{i\varphi}(\sigma )) d\sigma \geqslant \delta, \forall i=1,...,n\right\},
\end{equation*}  
where $\varphi :=(\varphi _{1},\ldots ,\varphi _{n})$ and $\tau _{0}=(\tau _{10},\ldots ,\tau _{n0})$.
\end{definition}
\begin{theorem} 
Let Assumptions \ref{ASS1.4} be satisfied. Let $\alpha >0$. For each $\delta \geqslant 0$ the subset $D_{\delta}$ is positively invariant under the semiflow generated by the system (\ref{EQ1.4}). Moreover for each $\delta >0$ and each initial conditions $(\varphi,\tau_0)\in D_{\delta}$ satisfying $\tau _{0}>0$, there exists $M(\delta)>0$ (independent of the initial conditions)  such that 
\begin{equation*}
\limsup _{t\rightarrow +\infty} \max_{i=1,\ldots ,n} \{A_i(t)\}\leqslant M(\delta).
\end{equation*} 
\end{theorem}
\begin{proof}
Suppose that $A_{i}(t)$ is the solution of system (\ref{EQ1.4}) for $t\geqslant 0$. By Theorem \ref{TH1.5}, they are bounded. Set
\begin{equation*}
M:=\max\left\{\limsup _{t\rightarrow +\infty}A_{1}(t),\dots ,\limsup _{t\rightarrow +\infty}A_{n}(t)\right\}.
\end{equation*}
Without loss of generality we might assume that $M=\limsup\limits _{t\rightarrow +\infty}A_{1}(t)$. Then we have the following alternative: \\
\textbf{Case 1:} There exists a time sequence $\{t_{n}\}_{n\in\mathbb{N}}$ which satisfies $\lim\limits_{n\rightarrow +\infty}t_{n}=+\infty$ and for any $t_{n}$, 
\begin{equation*}
A_{1}^{\prime}(t_{n})=0,
\end{equation*}
and 
\begin{equation*}
A_{1}(t_{n})\rightarrow M\text{ as } n\rightarrow +\infty .
\end{equation*}
Then we have
\begin{eqnarray*}
0 & =& A_{1}^{\prime}(t_{n})=-\mu _{A_{1}}A_{1}(t_{n})+\beta _{1}e^{-\mu _{J_{1}}\tau _{1}(t_{n})}\frac{f_{1}(Z_{1}(t_{n}))}{f_{1}(Z_{1}(t_{n}-\tau _{1}(t_{n})))}A_{1}(t_{n}-\tau _{1}(t_{n})) \\
& \leqslant & -\mu _{A_{1}}M+\beta _{1}\frac{f_{1}(\zeta _{11}M)}{f_{1}((\zeta _{11}+\dots +\zeta _{1n})M)}\lim_{n\rightarrow +\infty}A_{1}(t_{n}-\tau _{1}(t_{n})).
\end{eqnarray*}
So
\begin{equation*}
\lim_{n\rightarrow +\infty}A_{1}(t_{n}-\tau _{1}(t_{n}))\geqslant \frac{\mu_{A_{1}}M}{\beta _{1}}\cdot\frac{f_{1}(\zeta _{11}M)}{f_{1}((\zeta _{11}+\dots +\zeta _{1n})M)}\geqslant \frac{\mu_{A_{1}}M}{\beta _{1}}\cdot\frac{1}{M_{f_{1}}\left(\frac{\zeta _{11}+\dots +\zeta _{1n}}{\zeta _{11}}\right)}.
\end{equation*}
Then we get $\lim\limits_{n\rightarrow +\infty} A_{1}(t_{n}-\tau _{1}(t_{n}))\rightarrow +\infty$ as $M\rightarrow +\infty$. We assume that we can choose $M$ large enough such that $A_{1}(t_{n}-\tau _{1}(t_{n}))\geqslant m_{1}$. Since we have
\begin{eqnarray*}
\delta & \leqslant &\int_{\tilde{t}-\tau _{1}(\tilde{t})}^{\tilde{t}}f_{1}(Z_{1}(\sigma ))d\sigma \leqslant \int_{\tilde{t}-\tau _{1}(\tilde{t})}^{\tilde{t}}f_{1}(\zeta _{11}A_{1}(\sigma ))d\sigma \\
&\leqslant &\int_{\tilde{t}-\tau _{1}(\tilde{t})}^{\tilde{t}}f_{1}(\zeta _{11}z_{1}(\sigma -\tilde{t}+\tau _{1}(\tilde{t})))d\sigma =\int_{0}^{\tau _{1}(\tilde{t})}f_{1}(\zeta _{11}z_{1}(\sigma ))d\sigma ,
\end{eqnarray*}
then we can repeat a similar argument as before and we deduce the contradiction 
\begin{equation*}
0=\lim\limits _{n\rightarrow +\infty}A_{1}^{\prime}(t_{n})<0.
\end{equation*}
Thus there exists a certain constant $\tilde{M}$ such that $M<\tilde{M}$, which means that
\begin{equation*}
\limsup _{t\rightarrow +\infty}A_{i}(t)<\tilde{M},\forall i=1,\dots ,n.
\end{equation*}
\textbf{Case 2:} The solution $A_{1}(t)$ is eventually monotone. So we can assume that there exists a time $\bar{t}>0$ such that
\[
A_{1}^{\prime}(t)\geqslant 0, \forall t\geqslant \bar{t}
\]
(the case  $A_{1}^{\prime}(t)\leqslant 0$ being similar). Since $A_1(t)$ is eventually increasing, we deduce that 
\[
\lim_{t \to \infty }A_{1,t}=M \text{ in } C_{\alpha}:=\left\{ \phi \in C(-\infty ,0]:e^{-\alpha \vert
.\vert }\phi (.) \text{ is bounded}\right\}
\]  
where $C_{\alpha}$ is the Banach space endowed with the norm $\Vert \phi \Vert:= \Vert e^{-\alpha \vert .\vert }\phi (.)\Vert_{\infty} $.  

Since $A(t)$ is bounded, and $A_{1,t}^{\prime}$ is relatively compact in $X_{\alpha}$ (since $\alpha>0$, $A_{i}(t)$ satisfies the system (\ref{EQ1.4}) and by applying Arzel\`{a}-Ascoli theorem locally on the bounded interval $[-\theta^\ast,0]$  (for each $\theta^\ast>0$) and by using the step method to extend to $(-\infty ,0]$???), we have
\begin{equation*}
\lim _{t\rightarrow +\infty}A_{1,t}^{\prime}=0 \text{ in } L^{\infty}_{\alpha}:= \left\{ \phi :e^{-\alpha \vert .\vert }\phi (.) \in L^{\infty}(-\infty ,0]\right\}
\end{equation*}
where $L^{\infty}_{\alpha}$ is the Banach space endowed with the norm $\Vert \phi \Vert:= \Vert e^{-\alpha \vert .\vert }\phi (.)\Vert_{L^{\infty}} $. 
 
Furthermore we have
\[
\delta \leqslant \int_{t-\tau _{1}(t)}^{t}f_{1}(Z_{1}(\sigma ))d\sigma \leqslant \int_{t-\tau _{1}(t)}^{t}f_{1}(\zeta _{11}A_{1}(\sigma ))d\sigma 
\]
and by taking the limit when $t \to +\infty$ (and since by lemma \ref{LE4.1}  $t-\tau _{1}(t)\rightarrow  +\infty$) we obtain 
\[
\liminf_{t \to +\infty} \tau _{1}(t) \geqslant \frac{\delta }{f_{1}(\zeta _{11}M)}.
\]
So by taking the limit in the system (\ref{EQ1.4}) we obtain 
\[
0=\lim _{t\rightarrow +\infty}A_{1,t}^{\prime}\leqslant -\mu _{A_{1}}M+\beta _{1}e^{-\mu _{J_{1}}\frac{\delta }{f_{1}(\zeta _{11}M)}}\frac{f_{1}(\zeta _{11} M)}{f_{1}((\zeta _{11}+\dots +\zeta _{1n})M)}M.
\]
Now by using the same argument as before we deduce that this is impossible when $M$ is large enough.  
\end{proof}

\end{document}